\documentclass[11pt,oneside,letterpaper]{article} 
\usepackage[dvips,final]{graphicx}
\usepackage[dvips]{geometry}
\usepackage{setspace}
\usepackage{mathptmx}
\usepackage{times}
\usepackage{amsmath}
\usepackage{amsthm}
\usepackage{amssymb}
\usepackage{amsfonts}
\usepackage{mathrsfs}
\usepackage[all,cmtip]{xy} 
\usepackage{enumitem}
\usepackage{verbatim}
\usepackage[usenames,dvipsnames,svgnames,table]{xcolor}
\usepackage{flafter}
\usepackage{hyperref}
\usepackage{tocloft}

\DeclareFontFamily{U}{wncy}{}
\DeclareFontShape{U}{wncy}{m}{n}{<->wncyr10}{}
\DeclareSymbolFont{mcy}{U}{wncy}{m}{n}
\DeclareMathSymbol{\Sh}{\mathord}{mcy}{"58}                               
\newtheorem{theorem}{Theorem}[section] 
\newtheorem{lemma}[theorem]{Lemma}

\newtheorem{proposition}[theorem]{Proposition}
\theoremstyle{definition}

\newtheorem{remark}[theorem]{Remark}

\newcommand{\CH}{\mathrm{CH}}

\newcommand{\rank}{\mathrm{rank}}
\newcommand{\Gal}{\mathrm{Gal}}

\newcommand{\dimension}{\mathrm{dim}}
\newcommand{\cor}{\mathrm{cor}} 
\newcommand{\Tr}{\mathrm{Tr}} 
\newcommand{\Hom}{\mathrm{Hom}} 
\newcommand{\Frob}{\mathrm{Frob}}
\newcommand{\Ker}{\mathrm{Ker}}
\newcommand{\End}{\mathrm{End}}
\newcommand{\Aut}{\mathrm{Aut}}

\newcommand{\res}{\mathrm{res}}

\newtheorem*{rep@theorem}{\rep@title}
\newcommand{\newreptheorem}[2]{%
\newenvironment{rep#1}[1]{%
 \def\rep@title{#2 \ref{##1}}%
 \begin{rep@theorem}}%
 {\end{rep@theorem}}}
\newreptheorem{theorem}{Theorem}
\font\smallit=cmti10

\begin{document}

\begin{center}
\uppercase{\bf On the Selmer group attached to a modular form \\ and an algebraic Hecke character}
\vskip 20pt
{\bf Yara Elias\footnote{Supported by a doctoral scholarship of the Fonds Qu\'eb\'ecois de la Recherche sur la Nature et les Technologies.}}\\
{\smallit Department of Mathematics, McGill University}\\
{\tt yara.elias@mail.mcgill.ca}\\
\end{center}

\begin{abstract}

We construct an Euler system of generalized Heegner cycles to bound the Selmer group associated to a modular form and an algebraic Hecke character. The main argument is based on Kolyvagin's method adapted by Bertolini and Darmon \cite{bertolini1990descent} and by Nekov{\'a}{\v{r}} \cite{nekovar1992kolyvagin} while the key object of the Euler system, the generalized Heegner cycles, were first considered by Bertolini, Darmon and Prasanna in \cite{bertolini2013generalized}.
\end{abstract}

\tableofcontents

\vspace{20pt}

\textbf{Acknowledgments}. It is a pleasure to thank my advisor Henri Darmon for numerous discussions of the subject of this article as well as for his suggestions, corrections and valuable feedback on the writing of this monograph. I am grateful to Olivier Fouquet, Eyal Goren, Ariel Shnidman and the anonymous referees for many corrections and suggestions. This work was finalized at the Max Planck Institute for Mathematics, I am thankful for its hospitality.

\section{Introduction} \label{Introduction}

Kolyvagin \cite{kolyvagin1990grothendieck, gross1991kolyvagin} constructs an Euler system based on Heegner points and  uses it to bound the size of the Selmer group of certain (modular) elliptic curves $E$ defined over $\mathbb{Q}$, over imaginary quadratic fields $K$ assuming the non-vanishing of a suitable Heegner point. In particular, this implies that $$\mathrm{rank}(E(K))=1,$$ and the Tate-Shafarevich group $\Sh(E/K)$ is finite.
Bertolini and Darmon adapt Kolyvagin's descent to Mordell-Weil groups over ring class fields \cite{bertolini1990descent}. More precisely, they show that for a given character $\chi$ of $\Gal(K_c/K)$ where $K_c$ is the ring class field of $K$ of conductor $c$, $$\mathrm{rank}( E(K_c)^{\chi})=1$$ assuming that the projection of a suitable Heegner point is non-zero.
More generally, one can associate to a modular form $f$ of even weight $2r$ and level $\Gamma_0(N)$ a $p$-adic Galois representation $T_p(f)$ \cite{Jannsen1990mixed, scholl1990motives}.  
For a given number field $K$, there is a $p$-adic Abel-Jacobi map $$\Phi_{f,K}: \CH^r(X/K)_0 \longrightarrow H^1(K, T_p(f)),$$ where
\begin{itemize}
\item $X$ is the Kuga-Sato variety of dimension $2r-1$, that is, a compact desingularization of the $2r-2$-fold fibre product of the universal generalized elliptic curve over the modular curve $X_1(N)$,
\item $\CH^r(X/K)_0$ is the $r$-th Chow group of $X$ over $K$, that is the group of homologically trivial cycles on $X$ defined over $K$ of codimension $r$ modulo rational equivalence,
\item $H^1(K,T_p(f))$ stands for the first Galois cohomology group of $\Gal(\overline{K}/K)$ acting on $T_p(f)$.
\end{itemize} 
Nekov{\'a}{\v{r}} \cite{nekovar1992kolyvagin} adapts the method of Euler systems to modular forms of higher even weight to describe the image by the Abel-Jacobi map $\Phi_{f,K}$ of Heegner cycles on the associated Kuga-Sato varieties, hence showing that $$\dim_{\mathbb{Q}_p}(\Phi_{f,K}(e_f \ \CH^r(X/K)_0) \otimes_{\mathbb{Z}_p} \mathbb{Q}_p) =1$$ for a suitable projector $e_f$, assuming the non-vanishing of a suitable Heegner cycle.
In Article \cite{Elias2015Kolyvagin}, we combined these two approaches to study modular forms of higher even weight twisted by ring class characters $\chi$ of imaginary quadratic fields and showed that $$\dim_{\mathbb{Q}_p}(\Phi_{f, \chi, K}(e_{f, \chi} \ \CH^r(X/K)_0)\otimes_{\mathbb{Z}_p} \mathbb{Q}_p) =1$$ for a suitable projector $e_{f, \chi}$, assuming the non-vanishing of a suitable generalized Heegner cycle. 
In this article, we study the Selmer group associated to a modular form of even weight $r+2$ and an unramified algebraic Hecke character $\psi$ of $K$ of infinity type $(r,0)$. 
Our setting involves the \emph{generalized Heegner cycles} introduced by Bertolini, Darmon and Prasanna in \cite{bertolini2013generalized}.

Our motivation stems from the Beilinson-Bloch-Kato conjecture that predicts that $$\dimension_{\mathbb{Q}_p}(\Phi_{f,\psi, K}( e_{f,\psi} \ \CH^r(W/K)_0) \otimes_{\mathbb{Z}_p} \mathbb{Q}_p)=ord_{s=r+1 } L(f \otimes \theta_{\psi},s),$$ where $$\theta_{\psi} = \sum_{ \mathfrak{a} \subset \mathcal{O}_K} \psi( \mathfrak{a} ) q^{N( \mathfrak{a})}$$ is the theta series associated to $\psi$ \cite{Schneider1988introduction, Jannsen1990mixed}, $W$ is a Kuga-Sato like variety, and $e_{f, \psi}$ is a suitable projector.

Let $f$ be a normalized newform of level $\Gamma_0(N)$ where $N \geq 5 $ and even weight $r+2 > 2$.
Denote by $K=\mathbb{Q}(\sqrt{-D})$ an imaginary quadratic field 
with odd discriminant
satisfying the Heegner hypothesis, that is primes dividing $N$ split in $K$.
For simplicity, we assume that $$|\mathcal{O}_K^{\times}|=2.$$
Let $$\psi: \mathbb{A}_K^{\times} \longrightarrow \mathbb{C}^{\times}$$ be an unramified algebraic Hecke character of $K$ of infinity type $(r,0)$.
There is an elliptic curve $A$ defined over the Hilbert class field $K_1$ of $K$ with complex multiplication by $\mathcal{O}_K$.  
We further assume that $A$ is a $\mathbb{Q}$-curve, that is $A$ is $K_1$- isogenous to its conjugates in $\mathrm{Aut}(K_1)$. This is possible by the assumption on the parity of $D$, (see \cite[Section~11]{Gross1980arithmetic}).
Consider a prime $p$ not dividing $N D\phi(N )N_A N_\psi ,$ where $N_A$ is the conductor of $A$ and $N_\psi$ is the conductor of $\psi$. 
We denote by $V_{f}$ the $f$-isotypic part of the $p$-adic \'{e}tale realization of the motive associated to $f$ by Scholl \cite{scholl1990motives} and Deligne \cite{deligne1973modular} twisted by $\frac{r+2}{2}$ and by $V_\psi$ the $p$-adic \'{e}tale realization of the  motive associated to $\psi$ twisted by $\frac{r}{2}$. More precisely,
$V_\psi$ is the $\psi$-isotypic component of the first Galois cohomology group of $$ \mathrm{res}_{K_1/\mathbb{Q}}(A) = \prod_{\sigma \in \Gal(K_1/\mathbb{Q})} A^{\sigma}$$ where $A^{\sigma}$ is the $\sigma$-conjugate of $A$, (see Section \ref{MotiveConstruction2} for more details).  
Let $\mathcal{O}_F$ be the ring of integers of $$F=\mathbb{Q}(a_1,a_2, \cdots, b_1,b_2,\cdots),$$ where the $a_i$'s are the coefficients of $f$ and the $b_i$'s are the coefficients of the theta series $$\theta_{\psi} = \sum_{\mathfrak{a} \subset \mathcal{O}_K} \psi(\mathfrak{a}) q^{N(\mathfrak{a})}$$ associated to $\psi$.
Then $V_f$ and $V_\psi$ will be viewed (by extending scalars appropriately) as free $\mathcal{O}_{F}\otimes \mathbb{Z}_p $-modules of rank 2. We denote by $$V=V_f \otimes_{\mathcal{O}_{F} \otimes \mathbb{Z}_p } V_\psi$$ the $p$-adic \'{e}tale realization of the tensor product of $V_f$ and $V_\psi$
and let $V_{\wp}$ be its localization at a prime $\wp$ in $ F$  dividing $p$. Let $\mathcal{O}_{F, \wp}$ be the localization of $\mathcal{O}_{F}$ at $\wp$. Then $V_{\wp}$ is a four dimensional representation of $\Gal(\overline{\mathbb{Q}}/ \mathbb{Q})$ over $\mathcal{O}_{F, \wp}$ with coefficients in $$\End(\mathrm{Res}_{K_1,\mathbb{Q}}(A)/\mathbb{Q})= \oplus_{\sigma \in \Gal(K_1/\mathbb{Q})} \Hom(A, A^{\sigma}),$$
(see Section \ref{MotiveConstruction2}). 

By the Heegner hypothesis, there is an ideal $\mathcal{N}$ of $\mathcal{O}_K$ satisfying $$\mathcal{O}_K/\mathcal{N} = \mathbb{Z}/N \mathbb{Z}.$$ We can therefore fix $\Gamma_1(N)$ level structure on $A$, that is a point of exact order $N$ defined over the ray class field $L_1$ of $K$ of conductor $\mathcal{N}$. 
Consider a pair $(\varphi_1,A_1)$ where $ A_1$ is an elliptic curve defined over $K_1$ with level $N$ structure and 
$$ \varphi_1: A \longrightarrow A_1$$ is an isogeny over $\overline{K}$.
We associate to it a codimension $r+1$ cycle on $V$ $$\Upsilon_{\varphi_1}=Graph(\varphi_1)^r \subset (A \times A_1)^r \simeq (A_1)^r \times A^r $$ and define a \emph{generalized Heegner cycle} of conductor 1
$$\Delta_{\varphi_1} = e_{r} \Upsilon_{\varphi_1},$$  where $e_r$ is an appropriate projector \eqref{projector2}.
Then $\Delta_{\varphi_1}$ is defined over $L_1$. 
We consider the corestriction $$P(\varphi_1)=cor_{L_1,K}\Phi_{f,\psi,L_1}(\Delta_{\varphi_1}) \in  H^1(K, V_{\wp}/p )$$ where $\Phi_{f,\psi,L_1}$ is the $p$-adic \'{e}tale Abel-Jacobi map.
The Selmer group $$S \subseteq H^1(K, V_{\wp}/p )$$ consists of the cohomology classes which localizations at a prime $v$ of $K$ lie in
$$\left\{
   \begin{array}{ll}
         H^1(K^{ur}_{v}/K_v,V_{\wp}/p ) \hbox{ for } v \hbox{ not dividing }NN_A N_\psi p\\
         H^1_f(K_{v}, V_{\wp}/p ) \hbox{ for } v \hbox{ dividing }p
    \end{array}
\right.$$
where $K_v$ is the completion of $K$ at $v$, and $H^1_f(K_v, V_{\wp}/p )=H^1_{cris}(K_v, V_{\wp}/p )$ is the \emph{finite part} of $H^1(K_v, V_{\wp}/p )$ \cite{block1990lfunction}.
Note that the assumptions we make will ensure that $H^1(K^{ur}_{v}/K_{v},V_{\wp}/p)=0$ for $v$ dividing $NN_AN_\psi$. 
We denote by $Fr(v)$ the arithmetic Frobenius element generating $\Gal(K_v^{ur}/K_v),$ and by $I_v=\Gal(\overline{K_v}/K_v^{ur})$.

Let $h$ be the class number of $K$, and let $\mathfrak{a}_1, \cdots, \mathfrak{a}_h$ be a set of representatives of the ideal classes of $\mathcal{O}_K$. 
For a cohomology class $c$ in $H^1(K, V_{\wp}/p)$, we define $$c^{\psi}= \dfrac{1}{h} \ \sum_{i =1}^h  \psi^{-1}(\mathfrak{a}_i) \   \mathfrak{a}_i \cdot c $$
where $ \mathfrak{a}_i \cdot c$ is the image of $c$ by the map $H^1(K, V_{\wp}/p) \longrightarrow H^1(K, V_{\wp}/p \big/ V_{\wp}/p \ [\mathfrak{a}_i])$ induced by the projection map 
$$V_{\wp}/p \longrightarrow V_{\wp}/p \big/ V_{\wp}/p \ [\mathfrak{a}_i].$$ Then $(c^{\psi})^{\psi}=c^{\psi}$ and it is independent of the choice of representatives $\mathfrak{a}_1, \cdots, \mathfrak{a}_h$.
\begin{theorem} \label{theorem 2} 
Let $p$ be such that 
$$(p,N D\phi(N )N_A N_\psi)=1.$$
Suppose that $V_{\wp}/p$ is a simple $\Aut_{L_1}(V_{\wp}/p)$-module and suppose that $\Gal(\overline{L_1}/L_1)$ does not act trivially on $V_{\wp}/p$.
Suppose further that the eigenvalues of $Fr(v)$ acting on $V_{\wp}^{I_v}$ are not equal to 1 modulo $p$ for $v$ dividing $N N_A N_\psi$. 
Assume that $ P(\varphi_1)^{\psi} \neq 0$ in $ H^1(K,  V_{\wp}/p )^{\psi}$. 
Then the $\overline{\psi}$-eigenspace of the Selmer group $S^{\overline{\psi}}$ has rank 1 over $\mathcal{O}_{F, \wp}/p$.
\end{theorem}

To prove Theorem \ref{theorem 2}, we first consider the $p$-adic \'{e}tale realization of the twisted motive $V$ associated to $f$ and $\psi$ in the middle \'{e}tale cohomology of the associated Kuga-Sato variety. This provides us with a $p$-adic Abel-Jacobi map that lands in the Selmer group $S$. Next, we construct an Euler system of generalized Heegner cycles which were first considered by Bertolini, Darmon and Prasanna in \cite{bertolini2013generalized}. These algebraic cycles lie in the domain of the $p$-adic Abel Jacobi map.
In order to bound the rank of the $\overline{\psi}$-eigenspace of the Selmer group $S^{\overline{\psi}}$, we apply Kolyvagin's descent using local Tate duality, the local reciprocity law, an appropriate global pairing of $S$ and Cebotarev's density theorem.

Our development is an adaptation of Nekov{\'a}{\v{r}}'s techniques \cite{nekovar1992kolyvagin} and Bertolini and Darmon's approach \cite{bertolini1990descent}. The main novelty is that the algebraic Hecke character $\psi$ is of infinite type. In particular, the Galois representation associated to $V$ is four-dimensional over $\mathcal{O}_{F, \wp}/p$.

\section{Motive associated to a modular form and a Hecke character} \label{MotiveConstruction2}
In this section, we describe the construction of the four dimensional $\Gal(\overline{\mathbb{Q}}/\mathbb{Q})$-representation $$V_{\wp}=(V_{f} \otimes_{\mathcal{O}_{F}\otimes \mathbb{Z}_p} V_\psi)_{\wp},$$ where $\wp$ is a prime of $F$ dividing $p$.
Denote by $Y_1(N)$ the affine modular curve over $\mathbb{Q}$ parametrising elliptic curves with level $\Gamma_1(N)$.
Let $j: Y_1(N) \hookrightarrow X_1(N)$ be its proper compact desingularization classifying generalized elliptic curves of level $\Gamma_1(N)$.
The assumption $N\geq 5$ allows for the definition of the generalized universal elliptic curve $\pi: \mathscr{E} \longrightarrow X_1(N)$. 
Denote by $W_{r}$ the Kuga-Sato variety of dimension $r+1$, that is a compact desingularization of the $r$-fold fiber product of $\mathcal{E} $ over $X_1(N)$.
We let $W$ be the $2r+1$-dimensional variety defined by $$W= W_r \times A^r,$$
where $A$ is as in Section \ref{Introduction}. 
We denote by $[\alpha]$ the element of $\End_{K_1}(A)\otimes_{\mathbb{Z}} \mathbb{Q}$ corresponding to an element $\alpha$ of $K$.
Given two non-singular varieties $X$ and $Y$ over a number field, the group of correspondences $\mathrm{Corr}^m(X,Y)= \mathrm{CH}^{\dim(X)+m}(X \times Y)$ is the group of algebraic cycles of codimension $ \dim(X)+m$ on $X \times Y$ modulo rational equivalence, (see \cite[Section~2]{bertolini2012chow} for more details). 
Consider the projectors
$$e_A^{(1)}=\left( \dfrac{\sqrt{-D}+[ \sqrt{-D}] }{2\sqrt{-D}} \right) ^{\otimes r} +\left( \dfrac{\sqrt{-D}-[ \sqrt{-D}] }{2\sqrt{-D}} \right) ^{\otimes r}, \hbox{  } e_A^{(2)}= \left( \dfrac{1-[-1]}{2} \right) ^{\otimes r}$$
and $$e_A=e_A^{(1)} \circ e_A^{(2)}$$ in $ \mathbb{Q}[\End(A)]^r$. 
These projectors $e_A^{(1)}$, $ e_A^{(2)}$ and $e_A$ belong to the group of correspondences $$\mathrm{Corr}^0(A^r,A^r)_{\mathbb{Q}}=\mathrm{Corr}^0(A^r,A^r) \otimes_{\mathbb{Z}} \mathbb{Q}$$ from $A^r$ to itself. 
Let $$\Gamma_r=( \mathbb{Z}/N \rtimes \mu_2)^{r} \rtimes \Sigma_{r}$$ where $\mu_2=\{ \pm 1\}$ and $\Sigma_{r}$ is the symmetric group on $r$ elements. Then $\Gamma_r$ acts on $W_r$, (see \cite[Sections~1.1.0 \hbox{ and} 1.1.1]{scholl1990motives} for more details.)
The projector $e_W$ in $\mathbb{Z}\left[\dfrac{1}{2Nr!}\right][\Gamma_r]$ associated to $\Gamma_r$, called Scholl's projector, belongs to the group of zero correspondences $\mathrm{Corr}^0(W_r,W_r)_{\mathbb{Q}}$ from $W_r$ to itself over $\mathbb{Q}$, (see \cite[Section~2.1]{bertolini2012chow}).
Let 
\begin{align} \label{projector2}
e_{r} = e_W e_A,
\end{align}
be the projector in the group of zero correspondences $\mathrm{Corr}^0(W,W)_{\mathbb{Q}}$ from $W$ to itself over $\mathbb{Q}$.
We consider the sheafs $$\mathcal{F} =j_{*}Sym^{r}(R^1\pi_* \mathbb{Z}_p ) \ \hbox{ and } \ \mathcal{F}_{ A}=j_{*}Sym^{r}(R^1\pi_*  \mathbb{Z}_p) \otimes e_A H^r_{et}(\overline{A^r}, \mathbb{Z}_p) .$$
\begin{proposition}\label{Etale2}
The \'{e}tale cohomology group
$$H^1_{et}(X_1(N) \otimes \overline{\mathbb{Q}}, \mathcal{F}_{A}) $$ is isomorphic to $$ e_{r}  H_{et}^{2r+1}(\overline{W} \otimes \overline{\mathbb{Q}},  \mathbb{Z}_p)$$ and $$ H^1_{et}(X_1(N) \otimes \overline{\mathbb{Q}}, \mathcal{F}) \otimes e_A H^r_{et}(\overline{A^r}, \mathbb{Z}_p).$$
Also, we have $$e_{r} H_{et}^i(\overline{W} \otimes \overline{\mathbb{Q}},  \mathbb{Z}_p)=0$$ for all $i \neq 2r+1$.
\end{proposition}
\begin{proof}
The proof is a combination of \cite[theorem~1.2.1]{scholl1990motives} and \cite[proposition~2.4]{bertolini2013generalized}. Note that the proof in \cite[theorem~1.2.1]{scholl1990motives} involves $\mathbb{Q}_p$ coefficients but it is still valid in our setting, (see the Remark following \cite[Proposition~2.1]{nekovar1992kolyvagin}).
\end{proof}
Let $B=\Gamma_0(N)/ \Gamma_1(N)$. 
We define $$\widetilde{V}= e_B H^1_{et}(X_1(N) \otimes \overline{\mathbb{Q}},  \mathcal{F}_A)(r+1)$$
where 
$e_B=\frac{1}{|B|}{\sum_{b \in B} b}.$
Given rational primes $\ell$ coprime to $NN_A N_\psi$, the Hecke operators $T_{\ell}$ provide correspondences on $X_1(N)$ \cite{scholl1990motives}, inducing endomorphisms of $\widetilde{V}$.
Letting $$I = \Ker \{ \mathbb{T} \longrightarrow \mathcal{O}_F: T_{\ell} \mapsto a_{\ell}b_\ell,\ \forall  \ell \nmid NN_A N_\psi \},$$
we can define the $(f,\psi)$-isotypic component of $\widetilde{V}$ by $$V=\widetilde{V}/ I \widetilde{V}.$$ 
Hence, there is a map $m: \widetilde{V} \longrightarrow V$ that is equivariant under the action of Hecke operators $T_{\ell}$, for $\ell$ not dividing $NN_A N_\psi$ and under the action of the Galois group $\Gal(\overline{\mathbb{Q}}/\mathbb{Q})$.
The $f$-isotypic component of $e_B H^1_{et}(X_1(N) \otimes \overline{\mathbb{Q}}, \mathcal{F})(\frac{r}{2}+1)$ gives rise (by extending scalars appropriately) to $V_f$ and $e_A H^r_{et}(\overline{A^r}, \mathbb{Z}_p)(\frac{r}{2})$ gives rise to $V_\psi$. They are free $\mathcal{O}_{F}\otimes \mathbb{Z}_p$-modules of rank 2.
Hence, $$V_{\wp}=(V_{f} \otimes_{\mathcal{O}_{F} \otimes \mathbb{Z}_p} V_\psi)_{\wp}$$ is a four dimensional representation of $\Gal(\overline{\mathbb{Q}}/\mathbb{Q})$ over $\mathcal{O}_{F,\wp}$ with coefficients in $$\End(A/\mathbb{Q})= \oplus_{\sigma \in \Gal(H/\mathbb{Q})} \Hom(A, A^{\sigma}).$$

\section{p-adic Abel-Jacobi map}
We use Proposition \ref{Etale2} to view the $p$-adic \'{e}tale realization of the twisted motive $V$ associated to $f$ and $\psi$ in the middle \'{e}tale cohomology of the associated Kuga-Sato varieties.
For an integer $n$ with $(n,pNN_A N_\psi)=1$, let 
$$L_n=L_1 \cdot K_n $$ be the compositum of the ring class field $K_n$ of $K$ of conductor $n$ with the ray class field $L_1$ of $K$ of conductor $\mathcal N $.
Consider the $p$-adic \'{e}tale Abel-Jacobi map 
$$ \CH^{r+1}(W/L_n)_0 \longrightarrow H^1 \left(L_n,H^{2r+1}_{et} \left(W \otimes \overline{\mathbb{Q}} , \mathbb{Z}_p(r+1) \right) \right),$$ where $\CH^{r+1}(W/L_n)_0$ is the group of homologically trivial cycles of codimension $r+1$ on $W$ defined over $L_n$, modulo rational equivalence. This map factors through $e_{r}(\CH^{r+1}(W/L_n)_0 \otimes \mathbb{Z}_p)$ as the Abel-Jacobi map commutes with correspondences on $W$.
Composing the Abel-Jacobi map with the projectors $e_{r}$ and $e_B$ and with $m: \widetilde{V} \longrightarrow V$, we obtain a map 
$$\Phi_{f,\psi,L_n}: e_{r}(\CH^{r+1}(W/L_n)_0 \otimes \mathbb{Z}_p) \longrightarrow H^1(L_n,V),$$  
which is $\mathbb{T} \lbrack \Gal(L_n/\mathbb{Q}) \rbrack $-equivariant. 

\textbf{Conjectures and motivation.}
Beilinson \cite[Conjecture~5.9]{Beilinson1989height} predicts that $$ \dim_{\mathbb{Q}} e_r \CH^{r+1}(W/K)_0 = \mathrm{ord}_{s=r+1 } L(f \otimes \theta_{\psi},s).$$
Bloch and Kato \cite{block1990lfunction} conjecture that $$\Phi_{f,\psi,K}: e_r \CH^{r+1}(W/K)_0  \otimes_{\mathbb{Q}} \mathbb{Q}_p \longrightarrow S \otimes_{\mathbb{Z}_p} \mathbb{Q}_p $$ is an isomorphism.
As a consequence, one expects that
$$\dimension_{\mathbb{Q}_p}(S \otimes_{\mathbb{Z}_p} \mathbb{Q}_p)= \mathrm{ord}_{s=r+1 } L(f \otimes \theta_{\psi},s).$$
In certain settings to be relaxed in his forthcoming work, Shnidman \cite[Remark X.2]{Shnidman2014padic} relates $L'(f \otimes \theta_{\psi},r+1)$ to the complex valued height $<\cdot \ , \ \cdot>_{B}$ of $\Delta_{\varphi_1}$ defined by Beilinson \cite{Beilinson1989height}, that is 
$$ L'(f \otimes \theta_{\psi},r+1)=<\Delta_{\varphi_1} \ , \ \Delta_{\varphi_1}>_B $$ up to an explicit constant. However, it is worth noting that the pairing $<\cdot \ , \ \cdot>_{B}$ is not known to be non-degenerate.

Kolyvagin's results \cite{kolyvagin1990grothendieck} combined with those of Gross and Zagier \cite{gross1986heegner} prove the Birch and Swinnerton-Dyer conjecture for analytic rank less than or equal to 1. This is the particular case of the Beilinson-Bloch-Kato conjectures where the modular form $f$ is associated to an elliptic curve and $\psi$ is the trivial character.
Nekov{\'a}{\v{r}}'s results \cite{nekovar1992kolyvagin, nekovavr1995thep} that correpond to the setting where $\psi$ is trivial provide further evidence towards a $p$-adic analog of the Beilinson-Bloch-Kato conjecture of the form 
$$\dimension_{\mathbb{Q}_p}(\mathrm{Im}(\Phi_{f,K}) \otimes_{\mathbb{Z}_p} \mathbb{Q}_p)=\mathrm{ord}_{s=r+1 } L_p(f ,s)$$ due to Perrin-Riou \cite[section~2.8]{colmez1998fonctions},  \cite{perrinriou1987points}.
Shnidman \cite{Shnidman2014padic} relates the order of vanishing $$ \mathrm{ord}_{s=r+1} L_p(f \otimes \theta_{\psi},s)$$ of the $p$-adic $L$-function at $s=r+1$  to the height of the image by the $p$-adic Abel-Jacobi map of an appropriate generalized Heegner cycle of conductor 1.
In this article, we prove that $$\dimension_{\mathbb{Q}_p}(\mathrm{Im}(\Phi_{f,\psi,K})^{\overline{\psi}} \otimes_{\mathbb{Z}_p} \mathbb{Q}_p)=1.$$

\section{Generalized Heegner cycles}
We describe the construction of generalized Heegner cycles following Bertolini, Darmon and Prasanna \cite{bertolini2013generalized}. Consider pairs $(\varphi_i,A_i)$ where $ A_i$ is an elliptic curve with level $N$ structure and 
$$ \varphi_i: A \longrightarrow A_i$$ is an isogeny over $\overline{K}$.
Two pairs $$(\varphi_i,A_i), \ (\varphi_j,A_j)$$ are said to be \emph{isomorphic} if there is a $\overline{K}$-isomorphism $\alpha: A_i \longrightarrow A_j$ satisfying $\alpha \circ \varphi_i= \varphi_j$.
Recall that $\mathcal N$ is an ideal of $\mathcal{O}_K$ such that $\mathcal{O}_K/\mathcal{N} = \mathbb{Z}/N \mathbb{Z}.$
Let $Isog^{\mathcal{N}}(A)$ denote the isomorphism classes of pairs $(\varphi_i,A_i)$ with $ker(\varphi_i) \cap A[\mathcal{N}]$ trivial.
For $(\varphi_i, A_i)$ in $ Isog^{\mathcal{N}}(A)$, we associate a codimension $r+1$ cycle on $W$ $$\Upsilon_{\varphi_i}=graph(\varphi_i)^r \subset (A \times A_i)^r \simeq (A_i)^r \times A^r \subset W_r \times A^r$$
and define a \emph{generalized Heegner cycle} $$\Delta_{\varphi_i} = e_{r} \Upsilon_{\varphi_i}.$$
We view $graph(\varphi_i)$ as a subset of $A_i \times A$.
Denote by $D_{A_i}$ the element $$(graph(\varphi_i)- 0 \times A -\deg(\varphi_i)( A_i\times 0)) \hbox{ in } NS(A_i \times A),$$ where $NS(A_i \times A)$ is the N\'{e}ron-Severi group of $A_i \times A$. Then by an adaptation of \cite[Paragraph~II.3.2]{nekovavr1995thep}, we have $$\Delta_{\varphi_i}=D_{A_i}^r. $$
Let us assume that the index $i$ of $A_i$ indicates that $\End(A_i)$, which is an order in $\mathcal{O}_K$, has conductor $i$. 
Then $\Delta_{\varphi_i}$ is defined over the compositum of the abelian extension $\widetilde{K}$ of $K$ over which the isomorphism class of $A$ is defined, with the ring class field $K_{i}$ of conductor $i$. 
Since $\widetilde{K}$ is the smallest extension of $K_1$ over which $\Gal(\overline{K}/\widetilde{K})$ acts trivially on $A[\mathcal{N}]$, it is equal to the ray class field $L_1$ of $K$ of conductor $\mathcal{N}$.
Therefore, $\Delta_{\varphi_i}$ is defined over $$L_i=L_1 K_{i}.$$ 
Hence, $$\Delta_{\varphi_i} \hbox{ belongs to }\CH^{r+1}(W/L_i).$$ 
In fact, $\Delta_{\varphi_i}$ is homologically trivial on $W$ as shown in \cite[proposition~2.7]{bertolini2013generalized}.
In the rest of this section, we consider elements $(\varphi_i,A_i)$ and $ (\varphi_j,A_j)$ in $ Isog^{\mathcal{N}}(A).$

\begin{lemma} \label{Divisors2}
Consider the map $$g \times I: A_i \times A \longrightarrow A_j \times A,$$ where $g: A_i \longrightarrow A_j$ is an isogeny of elliptic curves and $I: A \longrightarrow A$ is the identity map. Then $$(g \times I )_{*} D_{A_i} = \deg(g)\ \dfrac{ \deg (\varphi_i)}{\deg (\varphi_j)} \  D_{A_j}.$$ 
\end{lemma}
\begin{proof}
We denote the intersection pairing of two divisors by a dot. We have
$$(g \times I )_{*} D_{A_i} \cdot (g \times I )_{*} D_{A_i} = \deg(g)^2 D_{A_i} \cdot D_{A_i},$$
where 
\begin{align*}
& D_{A_i} \cdot D_{A_i}\\
& = (\mathrm{graph}(\varphi_i)- 0 \times A -\deg(\varphi_i) A_i\times 0) \cdot (\mathrm{graph}(\varphi_i)- 0 \times A -\deg(\varphi_i) A_i\times 0) \\
& = \mathrm{graph}(\varphi_i)\cdot \mathrm{graph}(\varphi_i) + 0 \times A \cdot  0 \times A +  \deg(\varphi_i) A_i\times 0 \cdot \deg(\varphi_i) A_i\times 0\\
& -2 \mathrm{graph}(\varphi_i)\cdot 0 \times A -2 \mathrm{graph}(\varphi_i)\cdot \deg(\varphi_i) A_i\times 0 + 2  \deg(\varphi_i) A_i\times 0  \cdot 0 \times A\\
& = 0+0+0 -2 \deg^2(\varphi_i)-2 \deg(\varphi_i)+2 \deg(\varphi_i)\\
& =- 2 \deg^2 (\varphi_i).
\end{align*}
In the previous computation, the equality $\mathrm{graph}(\varphi_i)\cdot \mathrm{graph}(\varphi_i)=0$ follows from the implication 
$$(x, \varphi_i(x))=(x, \varphi_i(x)+P ) \implies P=0$$ for a translation of $\varphi_i(x)$ by some quantity $P$.
Hence, $$(g \times I )_{*} D_{A_i} \cdot (g \times I )_{*} D_{A_i}=-2 \deg(g)^2 \deg^2 (\varphi_i).$$
The induced map $$(g \times I )_{*}: NS(A_i \times A) \longrightarrow NS(A_j \times A)$$ respects the subgroups generated by complex multiplication cycles.
Then, since $$(g \times I )_{*} D_{A_i} = k D_{A_j},$$ where $A_j = g(A_i)$ and $k>0$,
we have $$(g \times I )_{*} D_{A_i} \cdot (g \times I )_{*} D_{A_i}= k^2 D_{A_j} \cdot D_{A_j}=-2k^2\deg^2 (\varphi_j).$$
The equality $-2 \deg(g)^2 \deg^2 (\varphi_i)= -2k^2\deg^2 (\varphi_j)$ then implies that $$k=\deg(g)\dfrac{ \deg (\varphi_i)}{\deg (\varphi_j)},$$ and $$(g \times I )_{*} D_{A_i} =\deg(g)\dfrac{ \deg (\varphi_i)}{\deg (\varphi_j)} D_{A_j}.$$
\end{proof}

\section{Euler system properties}
We study certain global and local norm compatibilities of generalized Heegner cycles satisfying the properties of Euler systems.
We have $\mathcal{O}_{F} \otimes \mathbb{Z}_p = \oplus_{\wp_i |p} \mathcal{O}_{F,\wp_i}$ where $\mathcal{O}_{F,\wp_i}$ is the completion of $\mathcal{O}_F$ at the prime $\wp_i$ dividing $p$.
Recall that $V_{\wp}=(V_{f} \otimes_{\mathcal{O}_{F}} V_\psi)_{\wp}$ where $\wp$ is a prime of $F$ dividing $p$.
For a Galois representation $V$, $$F (V )$$ will designate the smallest extension of $F$ such that $\Gal(\overline{F}/F(V ))$ acts trivially on $V$.
We denote by $\Frob_v(F_1/F_2)$ the conjugacy class of the Frobenius substitution of the prime $v \in F_2$ in $\Gal(F_1/F_2)$ and by $\Frob_\infty(F_1/\mathbb{Q})$ the conjugacy class of the complex conjugation $\tau$ in $\Gal(F_1/\mathbb{Q})$.
A rational prime $\ell$ is called a \emph{Kolyvagin prime} if
\begin{align}\label{Condition2}
(\ell,N D N_A N_\psi p )=1 \ \hbox{ and } \ a_\ell  b_\ell \ \equiv \ \ell+1 \ \equiv \ a_\ell^2-b_\ell^2+4 \ \equiv \ 0 \mod  p.
\end{align}
Let $$L=K(\mu_p)(V_{\wp}/p),$$
where $\mu_p$ is the group of $p$-th roots of unity.
Condition \eqref{Condition2} is equivalent to 
\begin{align} \label{Condition20}
\Frob_\ell \left(L/\mathbb{Q}\right)=\Frob_{\infty} \left(L/\mathbb{Q} \right).
\end{align}
Indeed, it is enough to compare the characteristic polynomial of the complex conjugation $(x^2-1)^2=x^4 -2x^2+1$ acting on $V_{\wp}/p$ with roots $-1$ and $1$, each with multiplicity 2, with the twist by $r+1 $ of the characteristic polynomial of the Frobenius substitution at $\ell$ acting on $V_{\wp}/p$ with roots $$\alpha_1 \alpha_3, \ \ \alpha_1 \alpha_4, \ \ \alpha_2 \alpha_3, \  \hbox{ and } \ \alpha_2 \alpha_4$$
 satisfying
$$\alpha_1 \alpha_2=\ell^{r}, \ \ \alpha_1+\alpha_2=b_\ell, \ \ \alpha_3 \alpha_4 = \ell^{r+1}, \ \ \alpha_3+\alpha_4 = a_{\ell}.$$ 
The characteristic polynomial of $\Frob(\ell)$ acting on $V_{\wp}/p$ is
\begin{align*}
& (x-\alpha_1 \alpha_3)(x - \alpha_1 \alpha_4)(x- \alpha_2 \alpha_3)(x- \alpha_2 \alpha_4) \\
= \ & (x^2 -(\alpha_1 \alpha_3 + \alpha_1 \alpha_4) x + \alpha_1^2 \alpha_3 \alpha_4)(x^2 - (\alpha_2 \alpha_3+\alpha_2 \alpha_4) x + \alpha_2^2 \alpha_3 \alpha_4)  \\
= \ &\left(x^2 - \alpha_1 a_\ell x +\ell^{r+1}\alpha_1^2 \right)\left(x^2 - \alpha_2 a_\ell x + \ell^{r+1} \alpha_2^2 \right)
\end{align*}
We use the equality $(\alpha_1+\alpha_2)^2=\alpha_1^2+\alpha_2^2+2 \alpha_1 \alpha_2$ that is $b_\ell^2 - 2 \ell^r =\alpha_1^2+\alpha_2^2$ to conclude that the latter equals
\begin{align*}
& x^4 - \left(\alpha_2 a_\ell + \alpha_1 a_\ell \right)x^3 + \left(\ell^{r+1}\alpha_1^2+\ell^{r+1} \alpha_2^2
+ \alpha_1 \alpha_2 a_\ell^2 \right)x^2 
\\ &
-\ell^{r+1}\left( \alpha_1 a_\ell \alpha_2^2+\alpha_2  a_\ell \alpha_1^2 \right)x  + \ell^{2r+2}\alpha_1^2 \alpha_2^2\\
= \ & x^4 - a_\ell b_\ell x^3 +\left( \ell^{r+1}b_\ell^2 -2 \ell^{2r+1}+ a_\ell^2 \ell^r \right)x^2 - \ell^{2r+1} a_\ell(\alpha_1+\alpha_2)x+\ell^{4r+2}\\
= \ & x^4 -a_\ell b_\ell  x^3 +\left(\ell^{r+1}b_\ell^2 -2 \ell^{2r+1}+ a_\ell^2 \ell^r \right)x^2 - \ell^{2r+1} b_\ell a_\ell x +\ell^{4r+2}.
\end{align*}
To twist this characteristic polynomial by $\ell^{r+1},$ it is enough to map $x \mapsto \ell^{r+1}x$. We obtain
\begin{align*}
& \ell^{4r+4} x^4 -a_\ell b_\ell  \ell^{3r+3}x^3 +\ell^{2r+2}\left(\ell^{r+1}b_\ell^2 -2 \ell^{2r+1}+ a_\ell^2 \ell^r \right)x^2 - \ell^{3r+2} b_\ell a_\ell x +\ell^{4r+2} \\
& =\ell^{4r+4}\left(x^4 -\dfrac{a_\ell b_\ell}{  \ell^{r+1}}x^3 +\dfrac{\ell^{r+1}b_\ell^2 -2 \ell^{2r+1}+ a_\ell^2 \ell^r }{\ell^{2r+2}}x^2 - \dfrac{ b_\ell a_\ell}{\ell^{r+2}}x +\dfrac{1}{\ell^{2}}  \right).
\end{align*}
On the one hand, using the congruences $$a_\ell  b_\ell \ \equiv \ \ell+1 \ \equiv \ a_\ell^2-b_\ell^2+4 \ \equiv \ 0 \mod  p,$$ we find that the characteristic polynomial $$x^4-2x^2+1$$ of the complex conjugation $\tau$ acting on $V_{\wp}/p$ is congruent to the characteristic polynomial of $\Frob(\ell)$ acting on $V_{\wp}/p$.
On the other hand, comparing the action of the Frobenius element $\Frob_\ell$ and the complex conjugation $\tau$ on $\zeta_p$, where $\zeta_p$ is a $p$-th root of unity, we obtain
$$\zeta_p^{\ell}= \Frob_{\ell} (\zeta_p) = \Frob_\infty(\zeta_p)= \zeta_p^{-1}.$$
This implies that $\ell \equiv -1 \mod p$. As a consequence, Condition \eqref{Condition2} is necessary to satisfy Equality \eqref{Condition20}.

Let $n=\ell_1 \cdots \ell_k$ be a squarefree integer where $\ell_i$ is a Kolyvagin prime for $i=1,\cdots,k$.  
The Galois group $\Gal(K_n/K_1)$ is the product over the primes $\ell$ dividing $n$ of the cyclic groups $$G_{\ell}=\Gal(K_{\ell}/K_1)$$ of order $\ell+1$. We denote by $\sigma_{\ell}$ a generator of $G_{\ell}$.
The extensions $L_1$ and $K_{n}$ are disjoint over $K_1$. Indeed, if $I= L_1 \cap K_n$ is a non-trivial extension of $K_1$, then $\Gal(I/K_1) = \prod_{\ell_{i_j}} \Gal(K_ {\ell_{i_j}}/K)$ for some $\ell_{i_j}$ dividing $n$. This implies that the primes $\ell_{i_j}$ ramify in $I$ and hence also in $L_1$, a contradiction since $(n,N)=1$.
Hence $$G_n = \Gal(L_n/L_1) \simeq \Gal(K_n/K_1).$$

The Frobenius condition on $\ell$ implies that it is inert in $K$. Denote by $\lambda$ the unique prime in $K$ above $\ell$. Writing $n$ as $n=\ell m$, we have that $\lambda$ splits completely in $L_m$ since it is unramified in $L_m$ and has the same image as 
$\Frob_{\infty}(L/K)=\tau^2=Id$ by the Artin map. A prime $\lambda_m$ of $L_m$ above $\lambda$ ramifies completely in $L_n$. We denote by $\lambda_n$ the unique prime in $L_n$ above $\lambda_m$.
Consider the image of $\Delta_{\varphi_n}$ by the Abel-Jacobi map $$\Phi_{f, \psi, L_n}: e_r (\CH^{r+1}(W/L_n)_0 \otimes \mathbb{Z}_p) \longrightarrow H^1(L_n,V).$$ 

\begin{proposition}\label{Hecke2}
Consider $(A_n, \varphi_n)\sim (A_m, \varphi_m) \in Isog^{\mathcal{N}}(A)$ where $n=\ell m$ for an odd prime $\ell$.
Then $$T_{\ell} \ \Phi_{f, \psi, L_m}(\Delta_{\varphi_m})= \cor_{L_n,L_m} \Phi_{f, \psi, L_n}(\Delta_{\varphi_n})= a_{\ell} b_\ell \Phi_{f, \psi, L_m}(\Delta_{\varphi_m}).$$
\end{proposition} 
\begin{proof}
By \cite[corollary~11.4]{Schoen1995computation}, $$T_{\ell} (\Delta_{\varphi_m})= \sum_{n_i } \Delta_{\varphi_{n_i} },$$ where the generalized Heegner cycles $\Delta_{\varphi_{n_i} }$ correspond to elements $$(A_{n_i}, \varphi_{n_i}) \sim (A_{m}, \varphi_{m})$$ in $Isog^{\mathcal{N}}(A)$ for elliptic curves $A_{n_i}$ that are $\ell$-isogenous to $A_m$ respecting level $N$ structure.
These elliptic curves $A_{n_i}$ correspond to $g A_m$ where $$g \in \Gal(L_n/L_m) \simeq \Gal(K_{n}/K_{m}) \simeq \Gal(K_{\ell}/K_1).$$
Hence $$\sum_{n_i } \Delta_{\varphi_{n_i} }= \sum_{g \in \Gal(L_n/L_m)} g \Delta_{\varphi_n} = \cor_{L_n,L_m} (\Delta_{\varphi_n})=a_{\ell} b_\ell \Delta_{\varphi_m},$$
where the last equality follows from the action of $T_{\ell}$ on $V$.
Finally, we apply the p-adic Abel-Jacobi map which commutes with $T_\ell$ to obtain $T_{\ell} \ \Phi_{f, \psi, L_m}(\Delta_{\varphi_m})= \cor_{L_n,L_m} \Phi_{f, \psi, L_n}(\Delta_{\varphi_n})$.
\end{proof}
For an element $c \in H^1(F,M)$, we denote by $\res_v(c) \in H^1(F_v,M)$ the image of $c$ by
the restriction map $H^1(F,M) \longrightarrow H^1(F_v,M)$ induced from the inclusion $$\Gal(\overline{F_v}/F_v) \hookrightarrow \Gal(\overline{F}/F).$$
\begin{proposition}\label{LocalRelation2}
Consider $(A_n, \varphi_n) \sim (A_m, \varphi_m) \in Isog^{\mathcal{N}}(A)$ where $n=\ell m$.
Then $$\res_{\lambda_n} \Phi_{f, \psi, L_n}(\Delta_{\varphi_n}) = k \ \Frob_{\ell}(L_n/L_m) \ \res_{\lambda_m} \Phi_{f, \psi, L_m}(\Delta_{\varphi_m}) $$ for $k=\ell  \ \dfrac{ \deg (\varphi_i)}{\deg (\varphi_j)}  $.
\end{proposition}
\begin{proof} 
Since $\lambda_m$ completely ramifies in $L_n$, we have $$\mathcal{O}_{L_m}/\lambda_m \simeq \mathcal{O}_{L_n}/\lambda_n,$$ which is isomorphic to the finite field with $\ell^2$ elements as $\ell$ is inert in $K$.
As a consequence, the reductions of the elliptic curves $A_n$ and $A_m$ at $\lambda_n$ and $\lambda_m$ are supersingular.
Hence the $\ell$-isogeny from $A_n$ to $ A_m$ reduces to the Frobenius morphism $\Frob_\ell$.
Therefore, we have $\Frob_\ell (A_m) \equiv  A_n \mod \lambda_n$.
By Proposition \ref{Divisors2}, this implies $$(\Frob_\ell \times I)_* D_{A_m} \equiv k \ D_{A_n} \mod \lambda_n$$ where $k=\ell \ \dfrac{ \deg (\varphi_i)}{\deg (\varphi_j)}  $ from which the result follows.
\end{proof}

\section{Kolyvagin cohomology classes}

We denote by $$\Phi_{f,\psi,L_n}(\Delta_{\varphi_n})_{\wp} \in H^1(L_n,V_{\wp})$$ the image of $\Phi_{f,\psi,L_n}(\Delta_{\varphi_n}) \in H^1(L_n,V)$ by the map $H^1(L_n,V) \longrightarrow H^1(L_n,V_{\wp})$ induced by the projection $V \rightarrow V_{ \wp}$. 
Let $$y_{\varphi_n} = \mathrm{red}(\Phi_{f,\psi,L_n}(\Delta_{\varphi_n})_{\wp}) \in H^1(L_n,V_{\wp}/p)$$ be the image of $\Phi_{f,\psi,L_n}(\Delta_{\varphi_n})_{\wp} \in H^1(L_n,V_{\wp})$ by the map $H^1(L_n,V_{\wp}) \longrightarrow H^1(L_n,V_{\wp}/p)$ induced by the projection $V_{\wp} \rightarrow V_{\wp}/p$.
We use certain operators \eqref{operator2} defined by Kolyvagin in order to lift the cohomology classes $y_{\varphi_n} \in H^1(L_n,V_{\wp}/p)$ to Kolyvagin cohomology classes $P(\varphi_n) \in H^1(K,V_{\wp}/p)$, for appropriate $n$.

\begin{lemma} \label{GaloisGroup2}
For all $n$, 
\begin{align*}
&H^0(L_n, V_{\wp}/p)=H^0(L_1,V_{\wp}/p)=0 \\ 
 \hbox{and} \ \ & \Gal(L_n(V_{\wp}/p)/L_n) \simeq \Gal(L_1(V_{\wp}/p)/L_1)
\end{align*}
\end{lemma}
\begin{proof}
The extensions $L_n/L_1$ and $L_1(V_{\wp}/p)/L_1$ are unramified outside primes dividing $n$ and $N_A N_\psi Np$. Therefore, $L_n \cap L_1(V_{\wp}/p)$ is unramified over $L_1$ and is hence contained in $L_1$, the maximal unramified extension of $K$ of conductor $\mathcal{N}$. Hence, $$H^0(L_n, V_{\wp}/p)=H^0(L_1,V_{\wp}/p).$$ 
The result follows from the assumption that $V_{\wp}/p$ is a simple $\Aut_{L_1}(V_{\wp}/p)$-module and $\Gal(\overline{L_1}/L_1)$ does not act trivially on $V_{\wp}/p$.
\end{proof}

\begin{proposition}
The restriction map $$res_{L_1,L_n}: H^1(L_1, V_{\wp}/p) \longrightarrow H^1(L_n, V_{\wp}/p)^{G_n}$$ is an isomorphism.
\end{proposition}
\begin{proof}
Recall that $G_n= \Gal(L_n/L_1)$ and let $G= \Gal(\overline{\mathbb{Q}}/L_n)$. The result follows from the inflation-restriction sequence
\begin{align*}
0 \rightarrow H^1(L_n/L_1,(V_{\wp}/p)^{G}) \xrightarrow{inf} H^1(L_1,V_{\wp}/p)
  \xrightarrow{res} H^1(L_n,V_{\wp}/p)^{G_n} \rightarrow H^2(L_n/L_1,(V_{\wp}/p)^{G}),
\end{align*}
since $H^0(L_n, V_{\wp}/p)=0$ by Lemma \ref{GaloisGroup2}.
\end{proof}
Let 
\begin{align} \label{operator2}
\Tr_{\ell}=\sum_{i=0}^{\ell} \sigma_{\ell}^i, \ \ \ D_\ell= \sum_{i=1}^{\ell} i \sigma_{\ell}^i \ \ \ \hbox{ in } \mathbb{Z}[G_\ell].
\end{align} 
They are related by 
\begin{align}\label{Relation2}
(\sigma_{\ell} -1) D_{\ell} = {\ell}+1-\Tr_{\ell}.
\end{align}
Define $$D_n= \prod_{\ell|n} D_{\ell} \in \mathbb{Z}[G_n].$$
\begin{proposition}\label{Fixed2}
$$D_n y_{\varphi_n} \in H^1(L_n,V_{\wp}/p)^{G_n}.$$ 
\end{proposition}
\begin{proof}
It is enough to show that for all $\ell$ dividing $n$, $$(\sigma_{\ell}-1)D_n y_{\varphi_n} =0.$$
We have $$(\sigma_{\ell}-1)D_n=(\sigma_{\ell}-1)D_{\ell}D_m=({\ell}+1-\Tr_{\ell})D_m,$$ where the last equality follows by Relation $\eqref{Relation2}$.
Since $res_{L_m,L_n} \circ \cor_{L_n,L_m} = \Tr_{\ell}$,
\begin{align*}
& (\ell+1-\Tr_{\ell})D_m \mathrm{red}(\Phi_{f,\psi,L_n}(\Delta_{\varphi_n})_{\wp}) \\
& = (\ell+1)D_m \mathrm{red}(\Phi_{f,\psi,L_n}(\Delta_{\varphi_n})_{\wp})-D_m a_{\ell} b_\ell \mathrm{red}(\Phi_{f,\psi,L_m}(\Delta_{\varphi_m})_{\wp}) \\
& \equiv 0 \hbox{ mod } p.
\end{align*}
by Proposition \ref{Hecke2} and Condition \ref{Condition2}. 
\end{proof}

As a consequence, the cohomology classes $D_n y_{\varphi_n} \in H^1(L_n, V_{\wp}/p)^{G_n}$ can be lifted to cohomology classes $c(\varphi_n) \in H^1(L_1,V_{\wp}/p)$ such that $$\mathrm{res}_{L_1,L_n} c(\varphi_n) =  D_n y_{\varphi_n}.$$ We define $$P(\varphi_n)= \cor_{L_1,K} c(\varphi_n) \hbox{ in } H^1(K, V_{\wp}/p).$$
For a place $v$ of $K$ and a cohomology class $c$ in $H^1(K, V_{\wp}/p)$, we denote by
$res_v(c)$ the image of $c$ by the map $H^1(K, V_{\wp}/p) \longrightarrow H^1(K_v, V_{\wp}/p)$ induced by the inclusion $G_{K_v} \hookrightarrow G_K $. 
\begin{proposition} \label{Ramification2}
Let $v$ be a prime of $L_1$. 
\begin{enumerate}[leftmargin=0cm,itemindent=.5cm,labelwidth=\itemindent,labelsep=0cm,align=left]
\item If $v |N_AN_\psi N$, then $\res_v(P(\varphi_n))$ is trivial. 
\item If $v \nmid N_A N_\psi N np$, then $\res_v(P(\varphi_n))$ lies in $H^1(K_{v}^{ur}/K_{v},V_{\wp}/p)$.
\end{enumerate}
\end{proposition}
\begin{proof}
\begin{enumerate}[leftmargin=0cm,itemindent=.5cm,labelwidth=\itemindent,labelsep=0cm,align=left]
\item
We denote by $$V_{\wp}/p^{dual}=\Hom(V_{\wp}/p,\mathbb{Z}/p\mathbb{Z}(1)  )$$ the local Tate dual of $V_{\wp}/p$.
The local Euler characteristic formula \cite[Section~1.2]{milne1986arithmetic} yields $$| H^1(K_v,V_{\wp}/p)|=|H^0(K_v,V_{\wp}/p)|\times |H^2(K_v,V_{\wp}/p)|.$$
Local Tate duality then implies $$| H^1(K_v,V_{\wp}/p)|=|H^0(K_v,V_{\wp}/p)|\times |H^0(K_v,V_{\wp}/p^{dual})|=|H^0(K_v,V_{\wp}/p)|^2$$ as $V_{\wp}/p$ is self-dual.
The Weil conjectures and the assumption on $Fr(v)$ imply that  $$((V_{\wp}/p)^{I_v})^{Fr(v)}=0$$ where $<Fr(v)> \ =\ \Gal(K_v^{ur}/K_v )$ and $I=\Gal(\overline{K_v}/K_v^{ur} )$ is the inertia group. 
Therefore, $$H^0(K_v,V_{\wp}/p)=((V_{\wp}/p)^{I_v})^{Fr(v)}=0.$$
\item
If $v$ does not divide $ N_A N_\psi N np$, then $$res_{L_{1,v},L_{n,v'}} \res_v c(\varphi_n)=\res_{v'} D_n y_{n} \in H^1(\overline{L_{n,v'}}/L_{n,v'},V_{\wp}/p)$$ for $v'$ above $v$ in $L_n$.
The exact sequence $$\cdots \xrightarrow{ } H^1(L^{ur}_{n,v'}/L_{n,v'},(V_{\wp}/p)^{I_v}) \xrightarrow{ } H^1(L_{n,v'},V_{\wp}/p)
\xrightarrow{\res} H^1(\overline{L_{n,v'}}/L^{ur}_{n,v'},V_{\wp}/p) \xrightarrow{ } \cdots $$ allows us to view the cohomology class $\res_{v'} D_n y_{n}$ that belongs to $\Ker(\res)$ as an element in $$ H^1(L^{ur}_{n,v'}/L_{n,v'},V_{\wp}/p).$$
Since $v$ is unramified in $L_n/L_1$, then $\res_v c(\varphi_n)$ belongs to $ H^1(L_{1,v}^{ur}/L_{1,v},V_{\wp}/p)$.
\end{enumerate}
\end{proof}

\section{Global extensions by Kolyvagin classes} 
We consider the restriction $d$ of an element $c$ of $H^1(K, V_{\wp}/p )$ to $H^1(L, V_{\wp}/p )$. Then $d$ factors through some finite extension $\tilde{L}$ of $L$. We denote by $$L(c)=\tilde{L}^{\ker(d)}$$ the subfield of $\tilde{L}$ fixed by $\ker(d)$. Note that $L(c)$ is an extension of $L$. 
In this section, we study extensions of $L$ by Kolyvagin cohomology classes $c$ and $P(\varphi_q)$, where $P(\varphi_q)$ will play a crucial role in the proof of Theorem \ref{theorem 2}.
We recall the statement of the theorem.
\begin{reptheorem}{theorem 2} 
Let $p$ be such that $$ (p,N D\phi(N )N_A )=1.$$
Suppose that $V_{\wp}/p$ is a simple $\Aut_{L_1}(V_{\wp}/p)$-module and suppose that $\Gal(\overline{L_1}/L_1)$ does not act trivially on $V_{\wp}/p$.
Suppose further that the eigenvalues of $Fr(v)$ acting on $V_{\wp}^{I_v}$ are not equal to 1 modulo $p$ for $v$ dividing $N N_A N_\psi$.
If $ P(\varphi_1)^{\psi}$ is non-zero in $H^1(K,V_{\wp}/p )^{\psi}$,
then the $\overline{\psi}$-eigenspace of the Selmer group $S^{\overline{\psi}}$ has rank 1 over $\mathcal{O}_{F, \wp}/p$.
\end{reptheorem}
Recall that $L=K(\mu_p)(V_{\wp}/p)$.
\begin{lemma} \label{dickson}
We have $$H^1(\Aut_{K}(V_{\wp}/p), V_{\wp}/p)=H^1(L/K, V_{\wp}/p)=0.$$
\end{lemma}
\begin{proof}
First note that if $p$ $\nmid |\Aut_{K}(V_{\wp}/p)|$, then $$H^1(\Aut_{K}(V_{\wp}/p), V_{\wp}/p)=0.$$
If $p $ divides $|\Aut_{K}(V_{\wp}/p)|$, then since $V_{\wp}/p$ is irreducible as an $\Aut_{K}(V_{\wp}/p)-$module, Dickson's lemma \cite[Theorem~6.21]{suzuki1982group} implies that $\Aut_{K}(V_{\wp}/p)$ contains $\mathrm{SL}_2(F_q)$ for some $q$. In particular, it contains $2I$ where $I$ is the identity map. 
Sah's lemma \cite[8.8.1]{lang1983fundamentals} states that if $G$ is a group, $M$ a $G$-representation, and $g$ an element of $\mathrm{Center}(G)$, then the map $x \longrightarrow (g-1) \ x$ is the zero map on $H^1(G,M)$.
Therefore, by Sah's lemma, the map $x \mapsto (2I- I)x=Ix$ is the zero map on $$H^1(\Aut_{K}(V_{\wp}/p), V_{\wp}/p)$$ implying that $H^1(\Aut_{K}(V_{\wp}/p), V_{\wp}/p)=0$.
As a consequence, $H^1(K(\mu_p)(V_{\wp}/p)/K, V_{\wp}/p)=0$ since $p$ does not divide $| \Gal(K(\mu_p)/K) | = p-1.$
\end{proof}

We denote the Galois group $\Gal(L/K)$ by $G$.
The restriction map $$r: H^1(K,V_{\wp}/p) \longrightarrow H^1(L,V_{\wp}/p)^G=\Hom_G(\Gal(\overline{\mathbb{Q}}/L ), V_{\wp}/p)$$ has kernel
$$\Ker(r)= H^1(L/K, V_{\wp}/p)=0$$ by Lemma \ref{dickson}.
Consider the evaluation pairing 
\begin{align} \label{global pairing 2}
[ \ , \ ] \ : \ r(S^{ }) \times \Gal(\overline{\mathbb{Q}}/L) \longrightarrow V_{\wp}/p.
\end{align}
We denote by $\Gal_S(\overline{\mathbb{Q}}/L)$ the annihilator of $r(S^{ })$. Let $L^S$ be the extension of $L$ fixed by $\Gal_S(\overline{\mathbb{Q}}/L)$ and $G_S$ the Galois group $\Gal(L^S/L)$.

\begin{remark}
The element $P(\varphi_1) $ belongs to $S$ by Proposition \ref{Ramification2}. Also, $L(P(\varphi_1))$ is a subfield of $L^S$. Indeed, assume $\rho \in \Gal_S(\overline{\mathbb{Q}}/L)$, then $[s,\rho]=0$ for all $s \in S^{ } $. Hence, $P(\varphi_1)$ defines a cocycle of $S^{ }$ by $$\rho \longrightarrow \rho( P(\varphi_1)) - P(\varphi_1)=0.$$ This implies that $\rho$ fixes 
$L(P(\varphi_1))$, a subfield of $L^S$.
\end{remark}

\begin{proposition}\label{Choice2}
There is a Kolyvagin prime $q$ such that $$res_{\beta} P(\varphi_1)^{\psi} \neq 0,$$ where $\beta$ is a prime dividing $q$ in $K$.
\end{proposition}
\begin{proof}
By Ceboratev's density theorem, there is a Kolyvagin prime $q$ such that $$\Frob_q(L^S/\mathbb{Q})= \tau h, \ \ h^{1+\tau} \notin \Gal(L^S/L(P(\varphi_1)^{\psi}))$$ for some $h$ in $\Gal(L^S/L)$. 
The restriction of $\tau h$ to $L$ is indeed $\tau$. Let $\beta$ be a prime of $L$ above $q$.
Since $\Frob_\beta(L(P(\varphi_1)^{\psi})/L)=(\tau h)^2$ does not fix $P(\varphi_1)^{\psi}$, $\beta$ is not in the kernel of the Artin map. Hence $\beta$ does not split completely in $L(P(\varphi_1)^{\psi}).$ Therefore, $res_{\beta} P(\varphi_1)^{\psi} \neq 0$.
\end{proof}
Consider the following extensions
\begin{displaymath}
    \xymatrix@R=3pc @C=0.5pc{  & H_0=H_1H_2 &  \\
	        H_1=L(P(\varphi_1)) \ar[ur] & & H_2=L(P(\varphi_{q})) \ar[ul] \\
                & L=K(\mu_p)(V_{\wp}/p) \ar[ur]_{V_2}  \ar[ul]^{V_1} \ar@{-->}[uu]^{V_0}  & }
\end{displaymath}
where $q$ is a Kolyvagin prime satisfying
$$res_{\beta} P(\varphi_1)^{\psi} \neq 0$$ for some $\beta$ in $L$ above $q$ (see Proposition \eqref{Choice2}).
We define $$V_i=\Gal(H_i/F) \ \ \hbox{ for } \ i=0,1,2.$$ 
An automorphism of $L(P(\varphi_1))/L $ corresponds to the multiplication of the generators of $L(P(\varphi_1))$ over $L$ by an element of $(V_{\wp})_p \simeq V_{\wp}/p$. Hence $$V_1=\Gal(L(P(\varphi_1))/L) \simeq V_{\wp}/p.$$
Similarly, we have $V_2 \simeq V_{\wp}/p$.
We recall that $h$ is the class number of $K$, and $\mathfrak{a}_1, \cdots, \mathfrak{a}_h$ are a set of representatives of the ideal classes of $\mathcal{O}_K$. 
For a cohomology class $c$ in $H^1(K, V_{\wp}/p)$, we defined $$c^{\psi}= \dfrac{1}{h} \sum_{i =1}^h \ \psi^{-1}(\mathfrak{a}_i) \  \mathfrak{a}_i \cdot c,$$ where $ \mathfrak{a}_i \cdot c$ is the image of $c$ by the map $H^1(K, V_{\wp}/p) \longrightarrow H^1(K, V_{\wp}/p \big/ V_{\wp}/p \ [\mathfrak{a}_i])$ induced by the projection map 
$$V_{\wp}/p \longrightarrow V_{\wp}/p \big/ V_{\wp}/p \ [\mathfrak{a}_i].$$ Note that $c^{\psi}$ is independent of the choice of representatives $\mathfrak{a}_1, \cdots, \mathfrak{a}_h$.
Furthermore, $(c^{\psi})^{\psi}=c^{\psi}$ lies in $H^1(K, V_{\wp}/p)^{\psi}$. 
We denote by $$H_1^{\psi}= L(P(\varphi_1)^{\psi}), \ \ H_1^{\overline{\psi}}= L(P(\varphi_1)^{\overline{\psi}}), \ \ H_2^{\psi}= L(P(\varphi_q)^{\psi}), \ \  H_2^{\overline{\psi}}= L(P(\varphi_q)^{\overline{\psi}}) ,$$
and we let $$V_i^{\psi}= \Gal(H_i^{\psi}/H_i), \ \ V_i^{\overline{\psi}}= \Gal(H_i^{\overline{\psi}}/H_i), \ \ \hbox{ for } \ i=1,2.$$
We will show that $$H_1^{\psi} \cap H_2^{\psi} = L.$$

\begin{lemma} \label{Corestriction2}
There is an isomorphism of $\mathcal{O}_{F,\wp}/p $-modules $$H^1(K^{ur}_{\lambda}/ K_{\lambda} ,V_{\wp}/p) \longrightarrow H^1(K^{ur}_{\lambda} ,V_{\wp}/p) $$ mapping $\res_{\lambda}P(\varphi_{m \ell})$
to $\res_{\lambda} P(\varphi_m)$.
Also, $\res_{\lambda} P(\varphi_\ell)$ is ramified.
\end{lemma}
\begin{proof}
This is an adaptation of \cite[Section~5]{Elias2015Kolyvagin} that uses the properties of the Euler system considered in Proposition \ref{Hecke2} and Proposition \ref{LocalRelation2}.
\end{proof}

Recall that $\tau$ denotes complex conjugation.
\begin{proposition}\label{Disjoint2}
The extensions $H_1^{\psi}$ and $H_2^{\psi}$ are linearly disjoint over $L$.
\end{proposition}
\begin{proof}
Linearly independent cocycles $c_1$ and $c_2$ of $H^1(K,V_{\wp}/p) $ over $\mathcal{O}_{F,\wp}/p$ can be viewed by Lemma \eqref{dickson} as linearly independent homomorphisms $h_1$ and $h_2$ in $\Hom_{\Gal(L/K)} (L,V_{\wp}/p)$ over $\mathcal{O}_{F,\wp}/p$.
Linearly independent homomorphisms $h_1$ and $h_2$ induce linearly disjoint extensions $\overline{L}^{\Ker(h_1)}$ and $\overline{L}^{\Ker(h_2)}$ of $L$.
Hence, if $H_1$ and $H_2$ were not linearly disjoint over $L$, we would have that 
$$u_1c_1+ u_2 c_2=0, \hbox{ for some $u_1, u_2$ in $(\mathcal{O}_{F,\wp}/p)^{*}$},$$
where $$c_1 = P(\varphi_1)^{\psi} \ \hbox{ and } \ c_2= P(\varphi_q)^{\psi}.$$
Lemma \eqref{Corestriction2} implies that $\res_{\beta} P(\varphi_{q})^{\psi}$ is ramified for $\beta \in K$ dividing $q$ and provides an isomorphism sending $\res_{\beta}P(\varphi_{q})^{\psi}$ to $\res_{\beta} P(\varphi_1)^{\psi}$. The latter implies by Proposition \eqref{Choice2} that $\res_{\beta} P(\varphi_{q})^{\psi}$ is non-zero.
This is inconsistent with the equality $u_1c_1+ u_2 c_2=0$ since the nontrivial cocycles $c_1,  c_2$ have different ramification properties.
\end{proof}

\section{Complex conjugation}
We study the action of complex conjugation on the image by the $p$-adic Abel-Jacobi map of generalized Heegner cycles. 
Recall that $D_{A_j}$ is the element $$(graph(\varphi_j)- 0 \times A -\deg(\varphi_j)( A_j\times 0)) \hbox{ in } NS(A_j \times A).$$ In this section, we write $D_{A_j, \varphi_j}$ instead of $D_{A_j}$ alone in order to keep track of the underlying map $\varphi_j$.

\begin{lemma}\label{ComplexConjugation2}
There is an element $\sigma$ in $\Gal(K_{j}/K)$ such that $$\tau  \Phi(\Delta_{\varphi_{j}})_{\wp}= (-1)^{\frac{r+2}{2}} \epsilon_L \ N^{r/2} \deg^{-r}( \sigma) \sigma \Phi(\Delta_{\varphi_{j}})_{\wp}, $$ where $ \epsilon_L$ is the sign of the functional equation of $L(f,s)$.
\end{lemma} 
\begin{proof}

We have that $(\tau \times I)_* (D_{A_j, \varphi_j}) =  D_{\tau(A_j), \overline{\varphi_j}}$.  Article \cite{gross1984heegner} shows that $\tau A_j =  W_N(\sigma A_j)$ for some $\sigma$ in $G(K_{j}/K),$ hence 
$$(\tau \times I)_* (D_{A_j, \varphi_j}^r)=  D_{\tau(A_j),\overline{\varphi_j}}^r=  D_{ W_N (\sigma A_j), W_N \circ \sigma \circ \varphi_j}^r.$$
Consider the map
$$ W \times I : W_N \times A \longrightarrow  W_N \times A: ((E,P),A) \longrightarrow ((E/ \langle P \rangle,P'),A),$$
where $P'$ is such that the Weil pairing $<P,P'>$ of $P$ with $P'$ satisfies $<P,P'>=\zeta_N$ for some choice $\zeta_N$ of an $N$-th root of unity $\zeta_N$.  
Note that $W$ has degree $N$.
Also, $$W_{*} f(\tau) d_{\tau} d_{z} = (-1)^{\frac{r+2}{2}} \epsilon_L \ N^{r/2} f(\tau)  d_{\tau} d_{z} .$$ 
This implies as in \cite[Proposition~6.2]{nekovar1992kolyvagin} that $$(W \times I )_{*} D_{W_N( \sigma A_j), W_N \circ \sigma \circ \varphi_j}^{r} =(-1)^{\frac{r+2}{2}} \epsilon_L \ N^{r/2} D_{W_N( \sigma A_j), W_N \circ \sigma \circ \varphi_j}^r,$$ 
while Proposition \ref{Divisors2} implies that the former equals $ N^r \dfrac{ \deg^r (\varphi_i)}{\deg^r (\varphi_j)}  D_{ \sigma A_j,\sigma \circ \varphi_j}^r. $
Hence,
$$D_{W_N( \sigma A_j),W_N \circ \sigma \circ \varphi_j}^r=  (-1)^{\frac{r+2}{2}} \epsilon_L \ k_1 \ N^{r/2} D_{ \sigma A_j,\sigma \circ \varphi_j}^r,$$ where $k_1=\dfrac{ \deg^r (\varphi_i)}{\deg^r (\varphi_j)}.$
Applying Proposition \ref{Divisors2} to the map $(\sigma \times I)$, we obtain $$(\sigma \times I)_*(D_{A_j,\varphi_j}^{r})=k_2 D_{\sigma A_j,\sigma \circ \varphi_j}^r,$$ where $k_2= \deg^r( \sigma) \dfrac{ \deg^r (\varphi_i)}{\deg^r (\varphi_j)} =\deg^r( \sigma) \  k_1$.
Hence, $$(\tau \times I)_*(D_{A_j,\varphi_{j}}^{r}) = D_{ W_N (\sigma A_j),W_N \circ \sigma \circ \varphi_j}^r=(-1)^{\frac{r+2}{2}} \epsilon_L \ N^{r/2} \ \deg^{-r}( \sigma)(\sigma \times I)_*(D_{A_j,\varphi_{j}}^{r}).$$
Therefore $$\tau  \Phi(\Delta_{\varphi_{j}})_{\wp}=(-1)^{\frac{r+2}{2}} \epsilon_L \ N^{r/2}\deg^{-r}( \sigma) \sigma \Phi(\Delta_{\varphi_{j}})_{\wp} .$$
\end{proof}

\begin{lemma} \label{nonvanishing}
Let $\epsilon=(-1)^{\frac{r+2}{2}} \epsilon_L $ and $k=\deg^{-r}( \sigma) \ N^{r/2}$. Then 
$$ \tau P(\varphi_{1})^{\psi}= \epsilon k \sigma P(\varphi_{1})^{\overline{\psi}} \ \hbox{ and }  \ \tau  P(\varphi_{q})^{\psi}= -\epsilon  k \sigma P(\varphi_{q})^{\overline{\psi}}.$$
\end{lemma}
\begin{proof}
For an element $c$ in $H^1(K, V_{\wp}/p)$, we have that
\begin{align*}
& \tau \cdot c^{\psi} =\dfrac{1}{h} \sum_{i =1}^h \ \overline{\psi}^{-1}(\mathfrak{a}_i) \ \tau \cdot (\mathfrak{a}_i \cdot c ) 
 =\dfrac{1}{h} \sum_{i =1}^h \ \overline{\psi}^{-1}(\mathfrak{a}_i) \ \tau (\mathfrak{a}_i \cdot c ) \tau^{-1} 
 = \dfrac{1}{h} \sum_{i =1}^h \ \overline{\psi}^{-1}(\mathfrak{a}_i) \ \mathfrak{a}_i \cdot \tau  c  \tau^{-1}
 = (\tau c)^{\overline{\psi}}.
\end{align*}
For $c=P(\varphi_j)$ and $\sigma$ in $\Gal(K_j/K)$, we have that
\begin{align*}
& \sigma \cdot c^{\psi} =\dfrac{1}{h} \sum_{i =1}^h \ \psi^{-1}(\mathfrak{a}_i) \ \sigma \cdot (\mathfrak{a}_i \cdot c ) 
 =\dfrac{1}{h} \sum_{i =1}^h \ \psi^{-1}(\mathfrak{a}_i) \ \sigma(\mathfrak{a}_i \cdot c ) \sigma^{-1} 
 = \dfrac{1}{h} \sum_{i =1}^h \ \psi^{-1}(\mathfrak{a}_i) \ \mathfrak{a}_i \cdot \sigma  c  \sigma^{-1}
= (\sigma c)^{\psi}.
\end{align*}
Hence $$\tau P(\varphi_{1})^{\psi}= (\tau P(\varphi_{1}))^{\overline{\psi}}= \epsilon k \sigma P(\varphi_{1})^{\overline{\psi}}.$$
The identity $\tau D_q =- D_q \tau$ indicates that
$$\tau  P(\varphi_{q})^{\psi}=  (\tau P(\varphi_{q}))^{\overline{\psi}} = -\epsilon  k \sigma P(\varphi_{q})^{\overline{\psi}}.$$
\end{proof}
For $(v_{11},v_{12},v_{21},v_{22})$ in $V_1^{\psi}  V_1^{\overline{\psi}} V_2^{\psi} V_2^{\overline{\psi}} $,  $$\tau (v_{11},v_{12},v_{21},v_{22})\tau = (\epsilon  k \sigma \tau v_{12},  \epsilon  k \sigma \tau v_{11},- \epsilon  k \sigma \tau v_{22},- \epsilon k \sigma \tau v_{21}).$$
We define $$U_0= \{  (v_{11},v_{12},v_{21},v_{22}) \ | \ \epsilon  k \sigma \tau v_{1i} + v_{1i} , - \epsilon  k \sigma \tau v_{2j} + v_{2j} ,\ i,j=1,2 \ \hbox{ generate } V_{\wp}/p  \} .$$
Then $U_0^{+}$ generates $V_0^{+}$.
We let $$L(U_0)=\{ \ell \hbox{ rational prime } | \ \Frob_{\ell}(H_0/\mathbb{Q})=[\tau u ], u \in U_0 \}.$$

\section{Local to Global study}
\textbf{Local Tate duality.}
Let $K_{\lambda}$ be a local field with residue field $F_q$ and let $A$ be a finite group with an unramified action of $\Gal(\overline{K_{\lambda}}/K_{\lambda})$ killed
by a prime $p$. 
Assume $p$ divides $q - 1$ so that $ \mu_p \subset K_{\lambda}$ and let $A' = \Hom(A,\mu_p)$.
We denote by $K_{\lambda}^t$, the maximal tamely ramified extension of $K_{\lambda}$, and by $H^1_{ur}(K_{\lambda},*),$ the group $H^1(K_{\lambda}^{ur}/K_{\lambda},*).$ 
The natural pairing $A \times A' \longrightarrow \mu_p$ yields the cup product pairing
$$ H^1( K_{\lambda} , A) \times H^1( K_{\lambda} , A') \longrightarrow H^2(K_{\lambda}, \mu_p) = \mathbb{Z}/p\mathbb{Z}$$
which induces a perfect local Tate pairing
$$H^1(K_{\lambda}^{ur}/K_{\lambda},A) \times  H^1( K_{\lambda} , A')/H^1(K_{\lambda}^{ur}/K_{\lambda},A) \longrightarrow \mathbb{Z}/p\mathbb{Z}.$$
Let $$\alpha: H^1(K_{\lambda}^{ur}/K_{\lambda},A) \xrightarrow{\sim} A/(\phi-1)A$$ be the evaluation map at the Frobenius element $\phi$ where $$Gal(K_{\lambda}^{ur}/K_{\lambda})= <\phi>.$$ Then $\alpha$ is an isomorphism.
The exact sequence of Galois groups $$0 \longrightarrow \Gal(\overline{K_{\lambda}}/K_{\lambda}^t) \longrightarrow \Gal(\overline{K_{\lambda}}/K_{\lambda}^{ur}) \longrightarrow \Gal(K_{\lambda}^t/K_{\lambda}^{ur}) \longrightarrow 0$$ induces the exact sequence
$$H^1(K_{\lambda}^t/K_{\lambda}^{ur},A') \longrightarrow H^1(K_{\lambda}^{ur},A') \longrightarrow H^1(K_{\lambda}^t,A') \longrightarrow 0,$$
where $H^1(K_{\lambda}^t,A')=0$ since $\Gal(\overline{K_{\lambda}}/K_{\lambda}^t)$ is a pro-$q$ group.  
Therefore, $$H^1(K_{\lambda}^{ur},A') \simeq H^1(K_{\lambda}^t/K_{\lambda}^{ur},A') \simeq \Hom(\mathbb{Z}/p\mathbb{Z}(1),A') \simeq \Hom(\mu_p, A').$$
Hence we have an isomorphism $$ H^1(K_{\lambda}^{ur},A') \xrightarrow{\sim}  \Hom(\mu_p, A') .$$
The exact sequence of Galois cohomology groups
$$0 \longrightarrow H^1(K_{\lambda}^{ur}/K, A') \longrightarrow H^1(K_{\lambda},A') \longrightarrow H^1(K_{\lambda}^{ur},A')^{\phi}\longrightarrow 0$$ allows us to identify  $H^1( K_{\lambda} , A')/H^1(K_{\lambda}^{ur}/K_{\lambda},A')$ with $$H^1(K_{\lambda}^{ur},A')^{\phi} \simeq \Hom(\mu_p, A')^{\phi}.$$
Hence, we obtain a perfect local pairing
$$\langle \ \cdot \ ,\ \cdot \ \rangle_{\lambda} \ : \  H^1(K_{\lambda}^{ur}/K_{\lambda},A) \times  H^1(K_{\lambda}^{ur},A')^{\phi} \longrightarrow \mathbb{Z}/p\mathbb{Z}.$$

\textbf{Set up of the proof.}
Given a Kolyvagin prime $\ell$, the Frobenius condition implies that it is inert in $K$. We denote by $\lambda$ the prime of $K$ lying above $\ell$.  
We have a perfect local pairing
$$\langle \ . \ , \ . \ \rangle_{\lambda}:H^1( K^{ur}_{\lambda}/K_{\lambda},(V_{\wp}/p)^{I_{\lambda}}) \times H^1(K_{\lambda}^{ur}, V_{\wp}/p) \longrightarrow \mathbb{Z}/p,$$ where $I_{\lambda}= \Gal(\overline{K_{\lambda}}/K_{\lambda}^{ur})$
and $\mathcal{O}_{F,\wp}$-linear isomorphisms
\begin{align} \label{TateIsomorphisms}
\{ H^1(K_{\lambda}^{ur}, V_{\wp}/p)\}^{dual} \simeq H^1(K^{ur}_{\lambda}/K_{\lambda},(V_{\wp}/p)^{I_{\lambda}}) \simeq (V_{\wp}/p)^{I_{\lambda}}/(\phi-1),
\end{align}
where $\phi$ generates $\Gal(K^{ur}_{\lambda}/K_{\lambda})$.
We denote by $$res_{\lambda}: H^1(K,V_{\wp}/p) \longrightarrow  H^1(K_{\lambda},V_{\wp}/p)$$ the restriction map from $H^1(K,V_{\wp}/p)$ to $H^1(K_{\lambda},V_{\wp}/p)$ induced by the embedding $$\Gal(\overline{K_{\lambda}}/K_{\lambda}) \hookrightarrow \Gal(\overline{K}/K).$$
Restricting the domain of $res_{\lambda}$ to the Selmer group $S$, we obtain
$$res_{\lambda}: S\longrightarrow  H^1(K^{ur}_{\lambda}/K_{\lambda},(V_{\wp}/p)^{I_{\lambda}}) .$$
Taking the $\mathbb{Z}/p\mathbb{Z}$-dual of this map, we obtain a homomorphism $$\omega_\ell:   H^1(K_{\lambda}^{ur}, V_{\wp}/p) \longrightarrow S^{dual}.$$
We denote $$X_{\ell} = \omega_\ell( H^1(K_{\lambda}^{ur}, V_{\wp}/p) )$$ in $S^{dual}$. 
We aim to bound the size of $S^{dual, \psi}$ by studying the set $$\{ X_\ell^{\psi} \}_{ \ell \in L(U_0)} .$$

\begin{proposition}\label{SDual2}
The modules $\{ X_{\ell} \}_{\ell \in L(U_0)}$ generate $S^{dual}$.
\end{proposition} 
\begin{proof}
Let $G=\Gal(L/K). $ Consider an element $s$ of $ S$. The restriction map
$$H^1(K, V_{\wp}/p) \xrightarrow{\res} H^1(L,V_{\wp}/p)^{G}$$
is injective by Lemma \eqref{dickson} as
$$ \Ker(\res)= H^1(L/K, V_{\wp}/p)=0.$$
We identify $s$ with its image by restriction in $$H^1(L,V_{\wp}/p)^G \subset \Hom_{G}(L, V_{\wp}/p).$$
We will show that $$\{ res_{\lambda} \}_{\ell \in L(U_0) }: S\longrightarrow \{ H^1(K^{ur}_{\lambda}/K_{\lambda},(V_{\wp}/p)^{I_{\lambda}})  \}_{\ell \in L(U_0) }$$ is injective. As a consequence, the induced map between the duals $$\{  H^1(K_{ \lambda}^{ur}, V_{\wp}/p) \}_{\ell \in L(U_0) } \longrightarrow  S^{dual}$$ is surjective. Hence, it is enough to show that $res_{\lambda}( s)=0$ for all $ \ell \in L(U_0)$ implies $ s=0$.
Consider $\widetilde{H_0}$, the minimal Galois extension of $\mathbb{Q}$ containing $H_0$ such that $s$ factors through $\Gal(\widetilde{H_0}/ L)$. Let $x$ be an element of $\Gal(\widetilde{H_0}/L)$ such that $x|_{H_0}$ belongs to $ U_0$.
By Cebotarev's density theorem, there exists $\ell$ in $L(U_0)$ such that $\Frob_{\ell}(\widetilde{H_0}/ \mathbb{Q})=[\tau x ].$
The hypothesis $res_{\lambda}(s)=0$ implies that $s(\Frob_{\lambda_L}(\widetilde{H_0}/L))=0$ for $\lambda_L$ 
above $\ell$ in $L$ since $\Frob_{\lambda_L}(\widetilde{H_0}/L)$ is a generator of
 $\Gal(\widetilde{H}_{0,\lambda_{\widetilde{H_0}}}/L_{\lambda_L})$, where
  $\lambda_{\widetilde{H_0}}$ is above $\lambda_L$ in $\widetilde{H_0}$. 
In fact,
$$\Frob_{\lambda_L}(\widetilde{H_0}/L)= (\tau x )^{|D(L/\mathbb{Q})| }=(\tau x)^2=x^{\tau}x=(x^{+})^2,$$
where $|D(L/\mathbb{Q})|$ is the order of the decomposition group $D(L/\mathbb{Q})$, also the order of the residue extension.
Therefore, $s(x^{+})=0$ for all $x \in \Gal(\widetilde{H_0}/L)$ such that $x|_{H_0}$ belongs to $U_0$.
Since $U_0^{+}$ generates $V_0^{+}$, we have that $s$ vanishes on $\Gal(\widetilde{H_0}/L)^{+}.$
Hence, $\mathrm{Im}(s)$ lies in $V_{\wp}/p^{-}$, the minus eigenspace of $V_{\wp}/p$ for the action of $\tau$.
In particular, it cannot be a proper $G$-submodule of $V_{\wp}/p$.
\end{proof}

\begin{proposition} \label{Generators2}
The elements $$res_{\lambda} P(\varphi_\ell)^{\psi} \hbox{ and } res_{\lambda} P( \varphi_{\ell q})^{\psi} $$ generate $ H^1(K_{\lambda}^{ur}, V_{\wp}/p)^{\psi}$.
\end{proposition}
\begin{proof} 
We have $$ H^1(K_{\lambda}^{ur}, V_{\wp}/p) \simeq  V_{\wp}/p(K_{\lambda})$$ by \eqref{TateIsomorphisms}. The module $ V_{\wp}/p(K_{\lambda})^{\psi}$ is of rank 2 over $\mathcal{O}_{F,\wp}/p $, hence, so is $ H^1(K_{\lambda}^{ur}, V_{\wp}/p)^{\psi}$.
The Frobenius condition on $\ell$ implies that 
$$g_1=res_{\lambda} P(\varphi_1)^{\psi} \hbox{ and } g_2=res_{\lambda} P( \varphi_{ q})^{\psi}$$ are linearly independent in $H^1(K_{\lambda}^{ur}, V_{\wp}/p)^{\psi}$ over $\mathcal{O}_{F,\wp}/p $. 
Indeed, if they were linearly dependent, then 
\begin{equation*}
 g_1^{(\tau x_{11})^2} - g_1, \hbox{ and } g_2^{(\tau x_{21})^2} - g_2
 \end{equation*}
where $\Frob_{\ell}(H_0/\mathbb{Q})=  \tau u= (\tau x_{11},\tau x_{12}, \tau x_{21}, \tau x_{22})$
would also be linearly dependent. The Frobenius condition implies that 
\begin{align*}
\{ \Frob_\ell(H_i/L)= x_{i1}^{\tau } x_{i1}= (\tau x_{i1})^2, \ \ i =1,2 \}
\end{align*}
generate $(V_{\wp}/p)$, which yields a contradiction since $(\tau x_{11})^2$ acts on $res_{\lambda} P(\varphi_1)^{\psi} $ which generates $H_1^{\psi}$ by $ g_1^{(\tau x_{11})^2} - g_1$ and $(\tau x_{21})^2$ acts on $res_{\lambda} P( \varphi_{ q})^{\psi}$ which generates $H_1^{\psi}$ by $g_2^{(\tau x_{21})^2} - g_2$. 
Therefore, by Lemma \eqref{Corestriction2}, $res_{\lambda}P(\varphi_\ell )^{\psi}$ and $res_{\lambda} P( \varphi_{\ell q})^{\psi}$ are linearly independent in $ H^1(K_{\lambda}^{ur}, V_{\wp}/p)^{\psi}$ over $\mathcal{O}_{F,\wp}/p $.

\end{proof}

\section{Reciprocity law and local triviality}
In this section, we use the local reciprocity law as well as the local properties of the Kolyvagin cohomology classes $P(\varphi_n)$ to study the modules $X_{\ell}$ for $\ell$ in $L(U_0)$.

\begin{proposition} \label{Reciprocity2}
We have $$ \sum_{\lambda|\ell|n} <s_{\lambda}^{\psi},res_{\lambda} P(\varphi_n)^{\psi}>_{\lambda} = 0.$$
\end{proposition}
\begin{proof}
The proof follows \cite[proposition~11.2(2)]{nekovar1992kolyvagin} where both the reciprocity law and the ramification properties of $P(\varphi_n)$ in proposition \ref{Ramification2} are used. 
\end{proof}

\begin{proposition}
\begin{enumerate}[leftmargin=0cm,itemindent=.5cm,labelwidth=\itemindent,labelsep=0cm,align=left]
\item The element $ \omega_{\ell}(res_{\lambda}  P(\varphi_{ \ell })^{\psi} )$ vanishes in $ X_{\ell}^{\psi}$ for $\ell$ in $L(U_0)$.\label{Vanishing2}
\item The modules $ \{ X_{\ell}^{\psi} \}_{ \ell \in L(U_0)}$ are generated over $O_{F,\wp}/p $ by $ \omega_q (res_{\beta} P( \varphi_{\ell_0 q})^{\psi})$.
\item The module $ S^{\overline{\psi}}$ is of rank 1 over $O_{F, \wp}/p $. 
\end{enumerate}
\end{proposition}
\begin{proof}
\begin{enumerate}[leftmargin=0cm,itemindent=.5cm,labelwidth=\itemindent,labelsep=0cm,align=left]
\item
Recall that $\lambda$ is the prime above $\ell$ in $K$. The image of $res_{\lambda}   P(\varphi_\ell)^{\psi} $ by the map $$\omega_{\ell}:  H^1(K_{\lambda}^{ur}, V_{\wp}/p)^{\psi} \longrightarrow X_{\ell}^{\psi}$$ is the homomorphism $$S \longrightarrow \mathbb{Z}/p\mathbb{Z} \ : \ s^{\psi} \mapsto \ < s_{\lambda}^{\psi},  P(\varphi_\ell)_{\lambda}^{\psi} >_{\lambda}.$$
Proposition \ref{Reciprocity2} implies that
$$ < s_{\lambda}^{\psi},  P(\varphi_\ell)_{\lambda}^{\psi} >_{\lambda}=0.$$ 
Hence, the image by $\omega_{\ell}$ of $res_{\lambda}   P(\varphi_\ell)^{\psi}$, 
one of the two generators of $ H^1(K_{\lambda}^{ur}, V_{\wp}/p)^{\psi}$ by proposition \ref{Generators2} vanishes.

\item
Let $\beta$ be the prime above $q$ in $K$.
Proposition \eqref{Reciprocity2} implies that $$ 
< s_{\lambda}^{\psi}, P( \varphi_{\ell q})_{\lambda}^{\psi} >_{\lambda}+<s_{\beta}^{\psi}, P( \varphi_{\ell q})_{\beta}^{\psi} >_{\beta}=0.$$
Hence, $$ \omega_{\ell}( res_{\lambda} P( \varphi_{\ell q})^{\psi})+\omega_{q}(res_{\beta} P( \varphi_{\ell q})^{\psi}) =0.$$
Therefore, $X_\ell^{\psi} $ is generated by $\omega_{q}(res_{\beta} P( \varphi_{\ell q})^{\psi})$ for all $\ell \in L(U_0)$. As a consequence, $$\{ X_\ell^{\psi} \}_{\ell \in L(U_0)} \ \subseteq \ X_{q}^{\psi},$$
where the rank one module $X_{q}^{\psi}$ is generated over $O_{F, \wp}/p $ by $\omega_{q}(res_{\beta} P( \varphi_{\ell_0 q})^{\psi})$ for some $\ell_0 $ in $L(U_0)$.

\item
By proposition \eqref{SDual2}, the set $\{ X_{\ell}^{\psi} \}_{\ell \in L(U_0)}$ generates $ S^{dual, \psi}$.
Furthermore, $P(\varphi_1)^{\overline{\psi}} $ belongs to $S^{\overline{\psi}}$ by Proposition \eqref{Ramification2} and is non-zero by the hypothesis on $ P(\varphi_1)^{\psi}$ and Proposition \eqref{nonvanishing}.
Therefore,
$$\rank(  S^{\overline{\psi}}) = \rank(   S^{dual, \psi})=1$$ over $O_{F, \wp}/p $.
\end{enumerate}
\end{proof}

\end{document}